\documentclass{amsart}

\usepackage{graphicx,mathrsfs}

\graphicspath{ {figures/} }

\newtheorem{theorem}{Theorem}[section]
\newtheorem{lemma}[theorem]{Lemma}

\theoremstyle{definition}
\newtheorem{definition}[theorem]{Definition}

\theoremstyle{remark}
\newtheorem{remark}[theorem]{Remark}

\numberwithin{equation}{section}

\newcommand{\R}{\mathbb{R}}
\newcommand{\Z}{\mathbb{Z}}

\newcommand{\PGL}{\mathrm{PGL}}

\newcommand{\SO}{\mathrm{SO}}
\newcommand{\Aut}{\mathrm{Aut}}

\begin{document}

\title[Manifolds with a cusp of non-diagonalizable type]{Convex projective manifolds with a cusp of any non-diagonalizable type}

\author{Martin D. Bobb}
\address{Department of Mathematics,••••••••
The University of Texas at Austin,
2515 Speedway, RLM 8.100,
Austin, TX 78712}

\begin{abstract}
Recent work of Ballas, Cooper, and Leitner identifies $(n+1)$ types of $n$-dimensional convex projective cusps, one of which is the standard hyperbolic cusp. Work of Ballas-Marquis, and Ballas-Danciger-Lee give examples of these exotic (non-hyperbolic) type cusps in dimension 3. Here an extension of the techniques of Ballas-Marquis shows the existence of all cusp types in all dimensions, except diagonalizable (type $n$). \cite{Ballas-Cooper-Leitner} \cite{Ballas-Marquis}
\end{abstract}

\maketitle

\let\thefootnote\relax

\section{Introduction}
Projective geometery, that is, the projective space $P(\R^{n+1})$ with its automorphism group $\PGL(n+1,\R)$ contains within it hyperbolic geometry. To identify hyperbolic space $\mathbb{H}^n$ as a subgeometry of projective geometry, select the interior of an ellipsoid properly contained in some affine chart of $P(\R^{n+1})$. Such a domain is preserved by a subgroup of $\PGL(n+1,\R)$ and may be equipped with a metric preserved by this subgroup. Up to scaling, this construction produces the Klein model of hyperbolic geometry.

If the ellipsoid in the previous construction is replaced with a projectively inequivalent open convex set which is properly contained in an affine chart, less familiar behavior appears. Such a domain is called a \textit{properly convex projective domain}, and has its own inherent geometry, which is generally less uniform than that of $\mathbb{H}^n$. In the 1960's Benzecri studied properly convex domains equipped with cocompact group actions, and was able to prove strong results about the geometric topology of these spaces. In the 2000's, Yves Benoist used Benzecri's results, along with careful analysis of important geometric substructures to prove a number of strong and surprising results. One of the most profound of these results is a geometric JSJ decomposition for compact quotients of irreducible properly convex domains in dimension 3. The quotient manifold is topologically a union of finite-volume hyperbolic $3$-manifolds, equipped with non-hyperbolic convex projective structures. In particular, the cusps of these manifolds have (virtually) diagonalizable holonomy, as opposed to the unipotent cuspidal holonomy of hyperbolic geometry. This brings to light natural questions about what structures the cusps of convex projective manifolds can support, and which ends of convex projective manifolds deserve to be called `cusps.'

The task of understanding convex projective cusps was initiated by Cooper, Long, and Tilmann \cite{Cooper-Long-Tillmann2}. Recently, Ballas, Cooper, and Leitner gave a complete classification of convex projective cusps with compact cross-section \cite{Ballas-Cooper-Leitner}. The result of this classification is a stratified space of structures, the $(n+1)$ strata in dimension $n$ referred to as \textit{types}, $t=0,\dots,n$, of convex projective cusps. In their chosen parametrization, these are identified with the strata of the positive closed dual Weyl chamber $\mathscr{W}^n$ in $(\R^{n})^*$. A point $\psi$ on a dimension-$t$ stratum of $\mathscr{W}^n$ represents a type $t$ cusp group. Cusp types will be discussed more completely in Section \ref{section: cusp neighborhoods}.

There is an obvious realization question for these cusp types. Type $0$ cusps are the cusps of hyperbolic geometry, and so are realized in every dimension as the ends of non-compact finite-volume hyperbolic manifolds. The other cusp types are not obviously realized. Precisely, the question is whether or not for any given type $t$ and dimension $n$ there exists an $n$-manifold $M$ which is homeomorphic to a compact core union some ends so that $M$ supports a convex projective structure in which all ends are cusps, and one is type $t$.

This question was answered in the affirmative for $t=1$ in all dimensions by Ballas and Marquis by deforming the hyperbolic structures on finite-volume hyperbolic manifolds \cite{Ballas-Marquis}. In this paper, we will extend this result to answer the realization question in the affirmative for every dimension and all but one cusp type, as well as every parameter $\psi$ within the realized types. The type that remains elusive is the diagonalizable type, $t=n$. Let $\partial\mathscr{W}^n$ be the union of positive codimension strata of $\mathscr{W}^n$. Precisely, the main theorem is 

\theoremstyle{theorem}
\newtheorem*{main theorem}{Theorem \ref{theorem: main theorem}}
\begin{main theorem}
For each $n > 1$ and each $\psi \in \partial\mathscr{W}^n$, there exists a connected convex projective manifold $M$ (with no boundary) of dimension $n$ which decomposes as the union of a compact manifold with boundary and a finite collection of generalized cusps, one of which has cusp parameter $\psi$ (and is hence of type $t(\psi)$).
\end{main theorem}

These examples are found as deformations of finite-volume hyperbolic structures on arithmetic manifolds. The realization question for diagonalizable cusps remains largely mysterious (except in dimensions two and three), and this paper will not address them, except to offer speculations. In dimension three, work of Ballas, Danciger, and Lee gives many examples of diagonalizable cusps  \cite{Ballas-Danciger-Lee}. Additionally, Ballas contributes examples of one-cusped 3-manifolds with cusps of types $1$ and $2$ \cite{Ballas}.

\section{Background}

\subsection{Convex projective domains}
A \textit{properly convex projective domain} is a subset  $\Omega \subset P(\R^{n+1})$ that is convex and has closure $\bar{\Omega}$ contained in some affine chart. Let $\Aut({\Omega})$ denote the subgroup of $\PGL(n+1,\R)$ that preserves $\Omega$.

We may equip $\mathring{\Omega}$ (the interior) with an $\Aut(\Omega)$ invariant metric in the following way. For two points, $x \neq y$ in $\Omega$, there is a unique projective line intersecting both, and two points, $z_1$ and $z_2$, in $\partial\Omega$ on this line. Suppose that these points are arranged in the order $z_1, x, y, z_2$. 

\begin{figure}[h]
\centering
\includegraphics[width=.3 \textwidth]{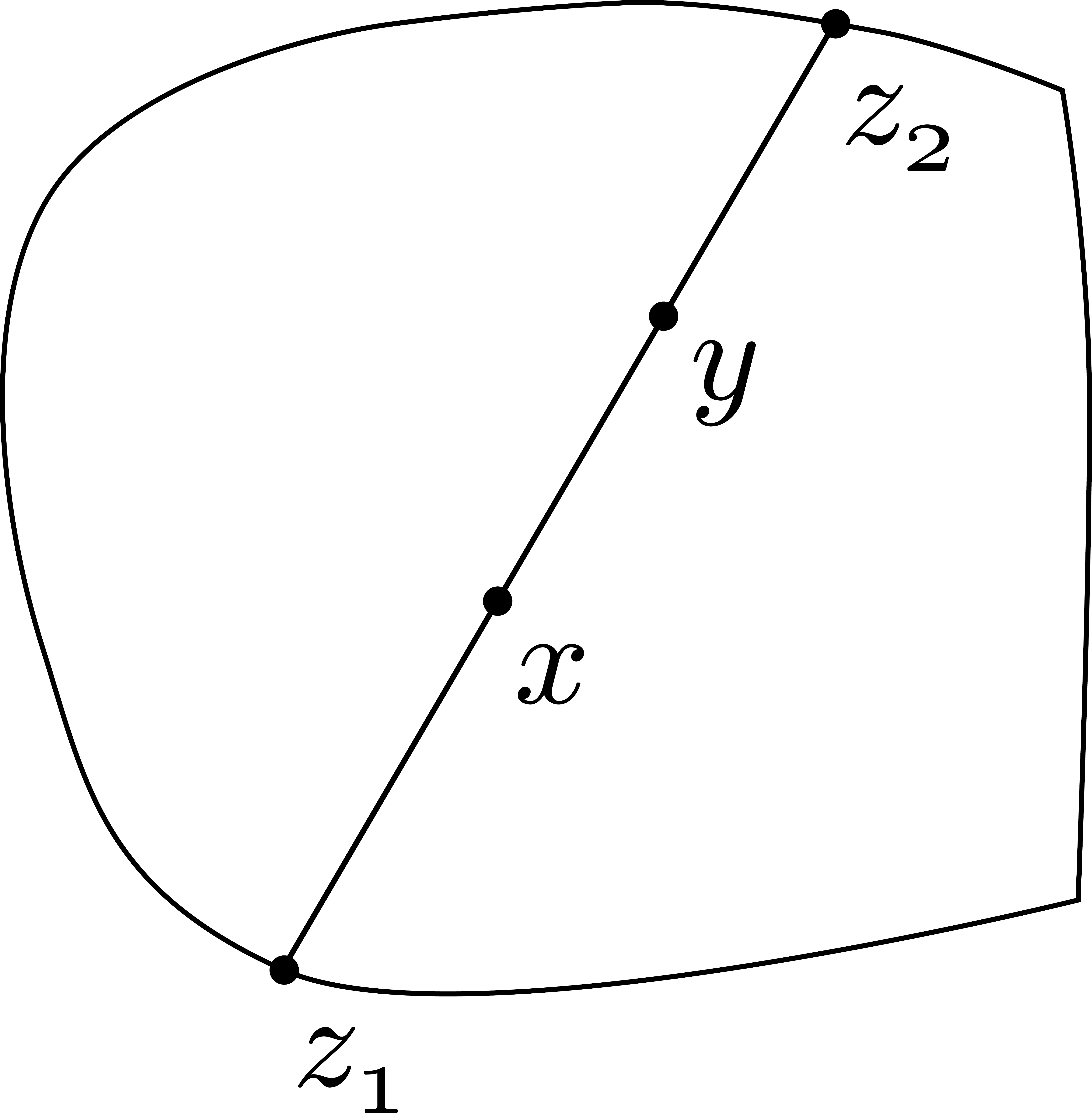}
\caption{A two-dimensional convex domain and the four points used to compute the distance from $x$ to $y$.}
\end{figure}

The line containing these four points is a copy of $P(\R^2)$ and after choosing a chart to $\R^1$ we may compute the following quantity, the \textit{cross ratio}:

$$[z_1:x:y:z_2]=\frac{|z_1-y||x-z_2|}{|z_1-x||y-z_2|}.$$

From this, we define the \textit{Hilbert metric} on $\Omega$ to be

$$d_\Omega(x,y) = \frac{1}{2}\log[z_1:x:y:z_2].$$

It is an exercise to show that the cross ratio is a projective invariant (it does not depend on the choice of chart for $P(\R^2)$), and that $d_\Omega$ is a metric whenever $\Omega$ is properly convex, or see \cite{delaHarpe}.

When $\Omega$ is an ellipsoid, $\Aut(\Omega)\cong \mathrm{PO}(n,1)$. Furthermore, $(\Omega,d_\Omega)\cong(\mathbb{H}^n,d_{\mathbb{H}^n})$ and $\Omega$ is the Klein model for hyperbolic geometry. Hence, convex projective geometry is a generalization of real hyperbolic geometry.

In general, the Hilbert metric on $\Omega$ is not induced by any Riemannian metric. However, it is possible to equip the tangent space to $\Omega$ with a smoothly varying norm (inducing the Hilbert metric), which makes $\Omega$ a Finsler manifold. Straight lines in projective space are always geodesic with respect to the Hilbert metric, though there may exist other geodesics if $\partial\Omega$ is not \textit{strictly convex}, that is, if $\partial \Omega$ contains some line segments.

\subsection{Separability}\label{subsection: separability general}
We require the technology of subgroup separability to produce examples of hyperbolic manifolds with a particular arrangement of codimension-1 submanifolds.


Over the past few decades, separability has developed into an essential tool for low-dimensional geometers. There are a number of equivalent definitions for subgroup separability, we give the one that clarifies the nomenclature.

\begin{definition}[Subgroup separability]\label{definition: separability}
A subgroup $H<G$ is \textit{separable} when for all $g \in G\setminus H$, there exists a subgroup $K$ so that $H<K<G$, $g\notin K$, and the index of $K$ in $G$ is finite. $G$ is then \textit{$H$-separable}.
\end{definition}

\begin{figure}[h]
\centering
\includegraphics[width=.4 \textwidth]{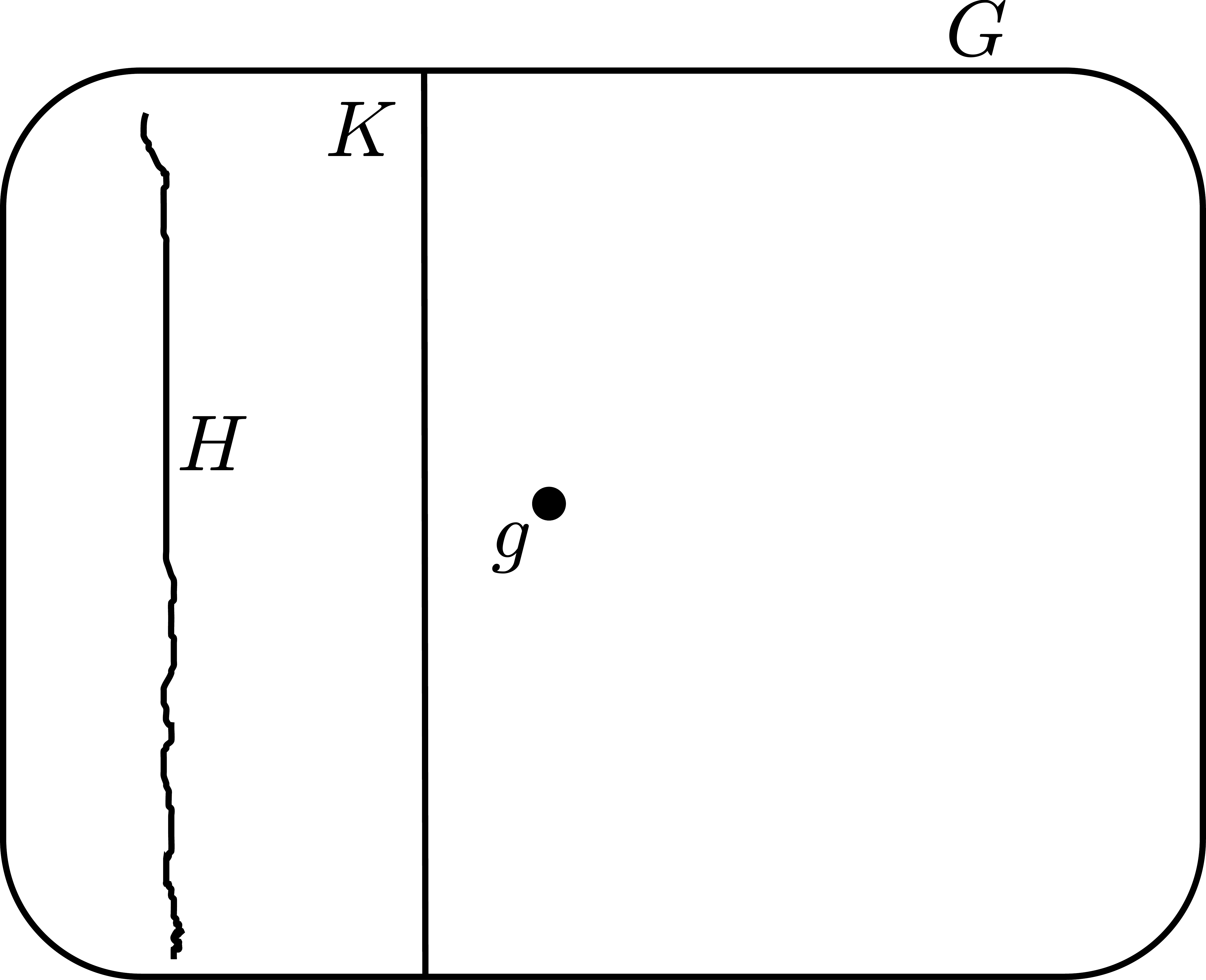}
\caption{Schematic of a group $G$ with subgroup $H$ which is separable from the element $g$ by the finite-index subgroup $K$.}
\end{figure}

A group that is 1-separable is called \textit{residually finite}, because any non-identity element can be represented nontrivially in a finite quotient. 

Many groups are separable on certain geometrically significant families of subgroups. When a group $G$ has the property that all of its subgroups with property $P$ are separable, we say that $G$ is \textit{$P$ Extended Residually Finite}. For example, Scott showed that surface groups are separable on finitely generated (local) subgroups, or Locally Extended Residually Finite (LERF) \cite{Scott} and Haglund and Wise showed that fundamental groups of compact special cube complexes are separable on quasi-convex subgroups, or Quasi-Convex Extended Residually Finite (QCERF) \cite{Haglund-Wise}.

\section{Cusp neighborhoods and groups}\label{section: cusp neighborhoods}

We briefly review the most relevant aspects of the results in Ballas, Cooper, and Leitner's paper \cite{Ballas-Cooper-Leitner}. Firstly, we borrow the following definition 

\begin{definition}\label{definition: generalized cusp}
A \textit{generalized cusp} is a properly-convex projective manifold $C=\Gamma \backslash \Omega$ with $\Gamma$ virtually abelian, $\partial C$ compact and strictly convex, and $C$ diffeomorphic to $\partial C\times\R_{\geq0}$. 
\end{definition}

Note that a usual hyperbolic cusp is a generalized cusp with this definition. Any appearance of the term \textit{cusp} refers to a generalized cusp.

Ballas, Cooper, and Leitner also contribute a classification of convex projective cusps, parametrizing the geometries of $n$-dimensional cusps with the \textit{positive closed dual Weyl chamber}
$$\mathscr{W}^n = \{\psi \in \mathrm{Hom}(\R^n,\R) ~\vert~ \psi(e_i)\geq\psi(e_{i+1})\geq0\ \forall i\}.$$

Any $\psi \in \mathscr{W}^n$ may be written as a row vector of non-increasing non-negative real numbers, $(\psi(e_1),\psi(e_2),\dots,\psi(e_n))$, and for brevity we write $\psi_i=\psi(e_i)$.

The \textit{cusp type} $t=t(\psi)$ is the greatest index $i$ in this representation so that $\psi_i>0$. For each parameter $\psi \in \mathscr{W}^n$, there is an associated model projective domain and group of automorphisms preserving that domain. Suppose that $t(\psi)<n$. Then the model domain $\Omega(\psi) \subset P(\R^{n+1})$ is foliated and given as a union of codimension one leaves

$$\Omega(\psi)=\bigsqcup_{c\geq0}\left\{\left[ (c-\sum_{i=2}^{t+1}\psi_i\log(x_i)+\frac{1}{2}\sum_{i=t+2}^n x_i^2),x_2,\dots,x_n,1\right] \bigg\vert x_2,\dots,x_{t+1} > 0 \right\}.$$

This domain is preserved (leaf-wise) by the \textit{type $t$ cusp group} (the translation subgroup of the cusp group in Ballas-Cooper-Leitner)

$$H(\psi) = \left\{\left( \begin{array}{cccc}
1 & 0 & ^t v & \sigma \\
0 & \boldsymbol{D} & 0 & 0 \\
0 & 0 & \boldsymbol{I} & v \\
0 & 0 & 0 & 1
\end{array} \right)\right\}$$
where $\boldsymbol{D}$ is a $t \times t$ diagonal matrix with positive entries $d_i$, $v$ is a vector of real numbers of length $(n-1-t)$, and $\boldsymbol{I}$ is the identity matrix of dimension $(n-1-t)$. The variable $\sigma$ in the matrix above denotes the quantity
$$\sigma = \frac{1}{2} \sum_{i=1}^{n-t-1} (v_i^2) - \sum_{j=1}^t \psi_j \log(d_j).$$ 

It is verifiable by a calculation that $H(\psi)$ preserves $\Omega(\psi)$, although it is not the full automorphism group of $\Omega(\psi)$. Let the reader be aware that we have exchanged the roles of the first and $(t+1)$st coordinates from \cite{Ballas-Cooper-Leitner}, because it will suit our needs better in Section \ref{section: iterated bending}. We may now define a $\psi$-cusp as a quotient of the model cusp neighborhood.

\begin{figure}
  \begin{minipage}[b]{0.3\textwidth}
    \includegraphics[width=\textwidth]{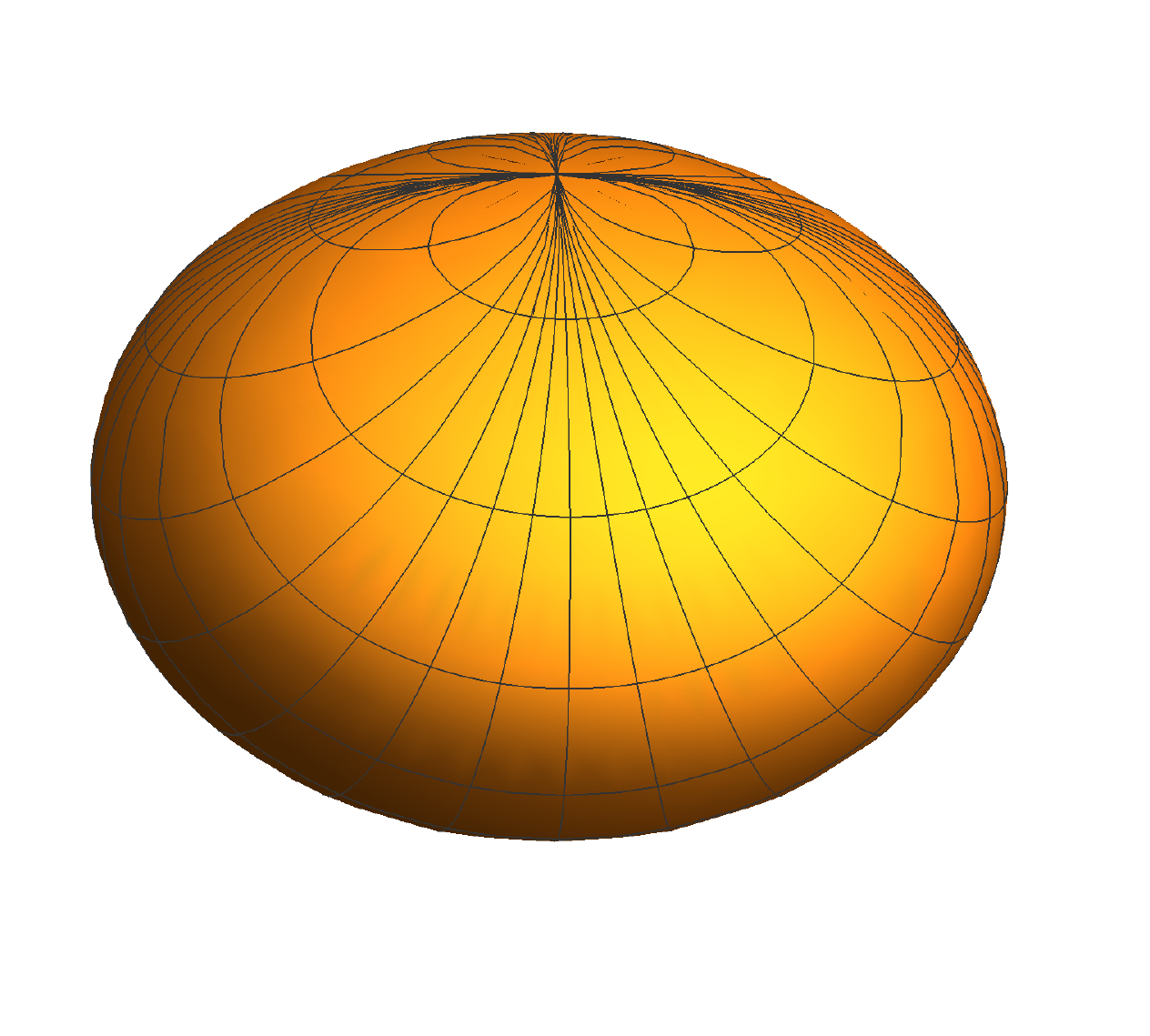}
  \end{minipage}
    \begin{minipage}[b]{0.3\textwidth}
    \includegraphics[width=\textwidth]{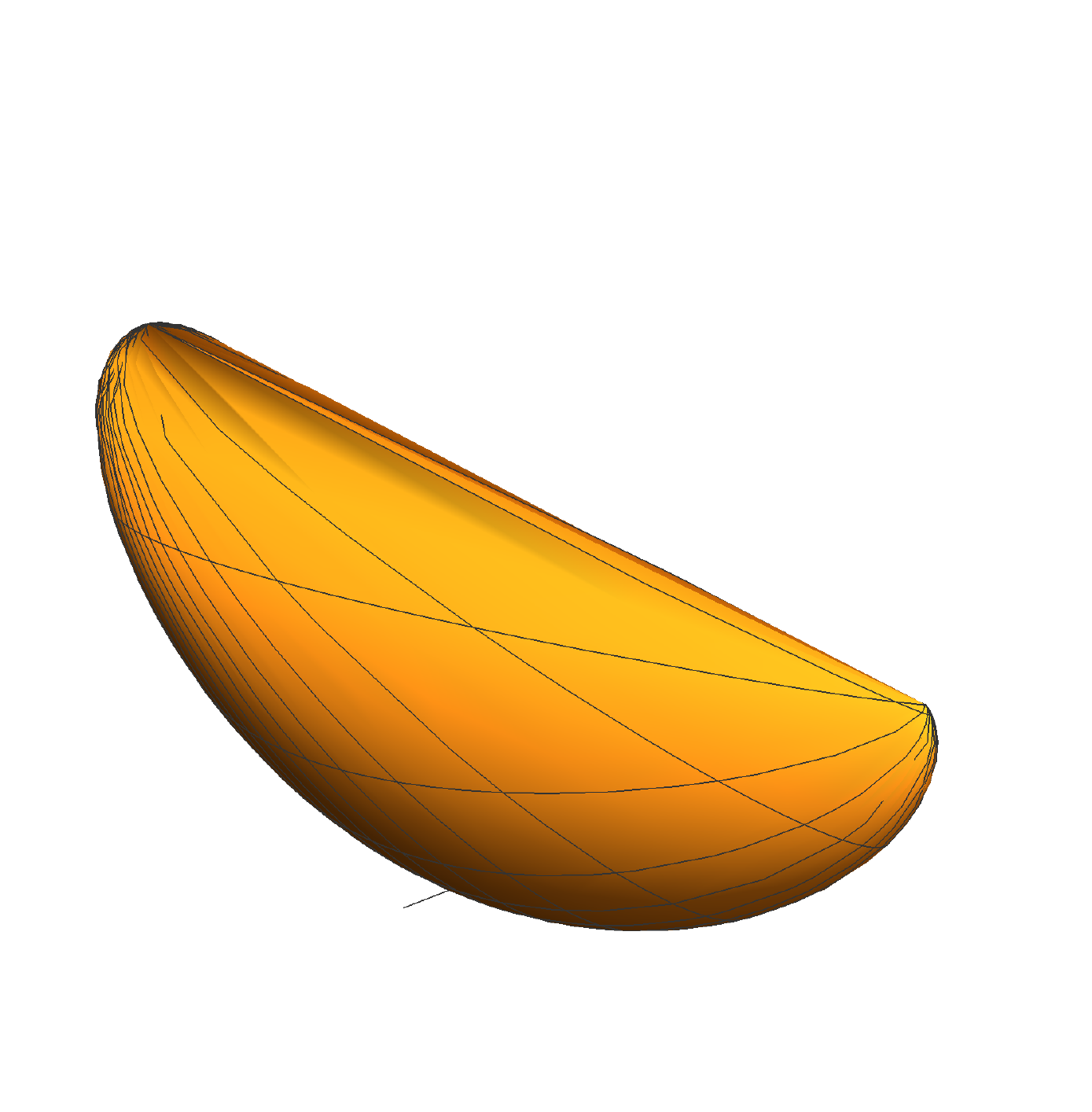}
  \end{minipage}
  \begin{minipage}[b]{0.3\textwidth}
    \includegraphics[width=\textwidth]{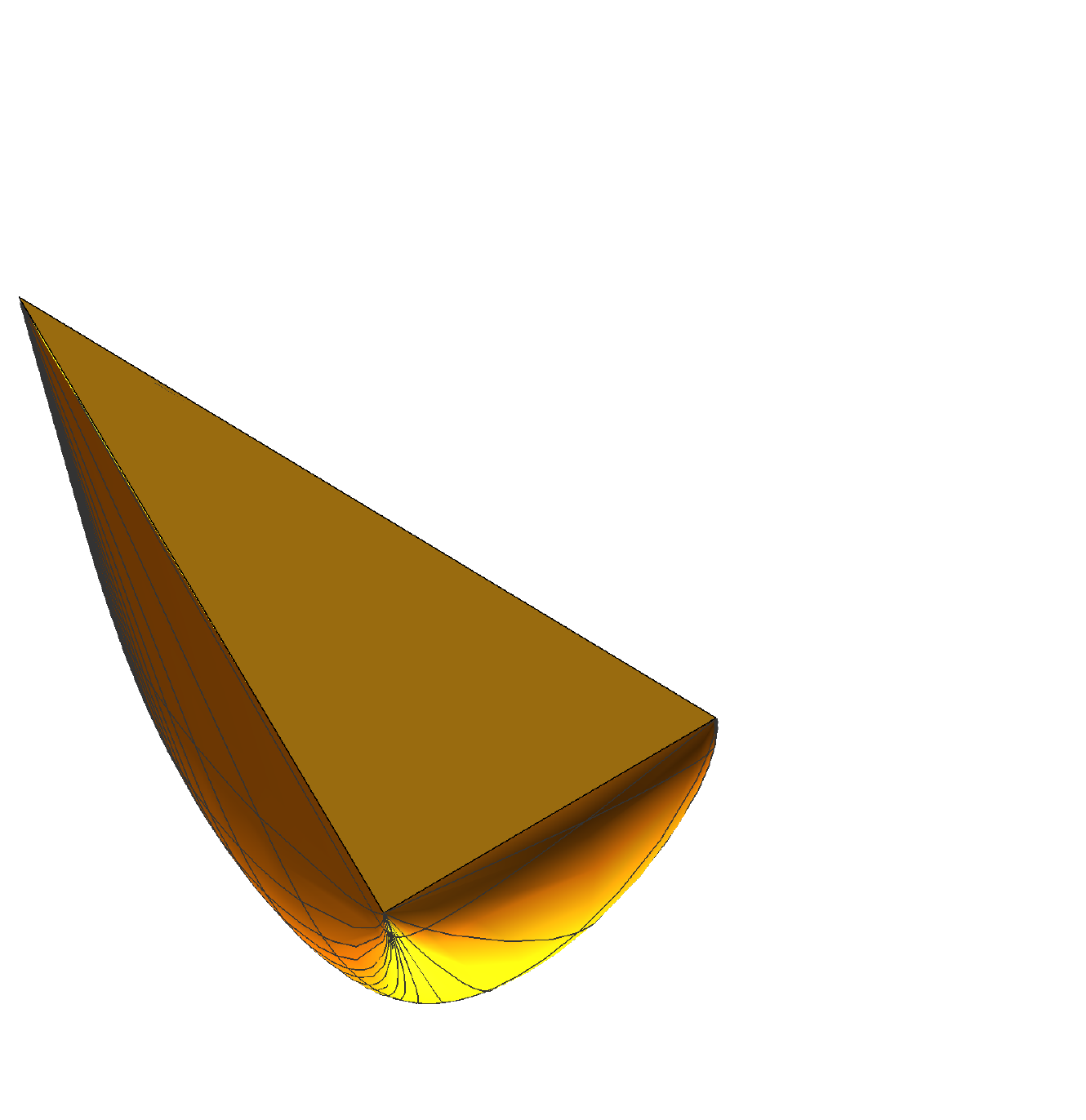}
  \end{minipage}
  \caption{Three-dimensionsal cusp domains of type 0, 1, and 2.}
\end{figure}

\begin{definition}\label{definition: cusp}
The quotient of $\Omega(\psi)$ by a lattice in $\Aut(\Omega(\psi))$ is a \textit{$\psi$-cusp} of \textit{type $t(\psi)$} with \textit{cusp parameter} $\psi$.
\end{definition}

There are generalized cusps which do not have torus cross-section, but from here forward, we only address torus cusps. As is shown by Ballas, Cooper, and Leitner, the boundary of $\Omega(\psi)$ supports a Euclidean structure, on which $\Aut(\Omega(\psi))$ acts by isometries. $H(\psi)$ is the translation subgroup of this Euclidean group, so every cusp is virtually a quotient of $\Omega(\psi)$ by a lattice in $H(\psi)$. That is, every generalized cusp virtually has torus cross-section.

Of course, there should be a notion of equivalence of cusps that allows for taking subneighborhoods that preserve the end. Intuitively, removing some compact portion of a cusp should not count as a different cusp. For a precise treatment of this equivalence, we refer the reader to \cite{Ballas-Cooper-Leitner}. For our purposes, it will be sufficient to rely on the following theorem from the same.

\begin{theorem}[Ballas-Cooper-Leitner, Theorem 0.2 \cite{Ballas-Cooper-Leitner}]\label{theorem: B-C-L}\
\begin{enumerate}
\item If $\Gamma$ and $\Gamma'$ are lattices in $H(\psi)$, the following are equivalent
\begin{itemize}
\item $\Gamma \backslash \Omega(\psi)$ and $\Gamma' \backslash \Omega(\psi)$ are equivalent as generalized cusps.
\item $\Gamma$ and $\Gamma'$ are conjugate in $\PGL(n+1,\R)$.
\item $\Gamma$ and $\Gamma'$ are conjugate in $\Aut(\Omega(\psi))$.
\end{itemize}
\item A lattice in $H(\psi)$ is conjugate in $\PGL(n+1,\R)$ into $H(\psi')$ iff $H(\psi)$ is conjugate to $H(\psi')$.
\item $H(\psi)$ is conjugate in $\PGL(n+1,\R)$ to $H(\psi')$ iff $\psi'=r \psi$ for some $r>0$.
\end{enumerate}
\end{theorem}

The theorem in \cite{Ballas-Cooper-Leitner} says more than this, but we record here only what is necessary for this paper. 

A main result of Ballas, Cooper, and Leitner is the following uniformization theorem.

\begin{theorem}[Ballas-Cooper-Leitner, Theorem 0.1 \cite{Ballas-Cooper-Leitner}]\label{theorem: B-C-L uniformization}\
Every generalized cusp is equivalent to a $\psi$-cusp.
\end{theorem}

Hence, we will treat the notion of a generalized cusp and a $\psi$-cusp as interchangeable.

\section{Iterated bending}\label{section:iterated bending}

\subsection{Bending representations}\label{subsection: what is bending}

Bending is an algebraic construction motivated by geometric structures. For a review of bending techniques in the context of convex projective geometry, we refer the reader to \cite{Ballas-Marquis} and \cite{Marquis}. This technique was first described in a general context by Johnson and Millson \cite{Johnson-Millson}, though the idea is due to Thurston. Throughout we assume basic understanding of $(G,X)$-structures, for a reference see \cite{Thurston}.

Suppose that $\rho_0:\Gamma\rightarrow G$ is a representation of some discrete group $\Gamma$ into a Lie group $G$, and that $\Gamma = \Gamma_1 *_S \Gamma_2$ is a free product with amalgamation. When the centralizer $C(\rho_0(S))<G$ is non-trivial, we may then deform $\rho_0$ in the following way. Choose a path $c_t$ into $C(\rho_0(S))$ with $c_0=1$. Let

$$\rho_t(\gamma)=
   \begin{cases} 
      \rho_0(\gamma) & \gamma\in\Gamma_1 \\
      c_t\rho(\gamma)c_t^{-1} & \gamma\in\Gamma_2 
   \end{cases}.
$$

Because $\Gamma$ is a free product with amalgamation, any element of $\Gamma$ can be written as a product of elements from $\Gamma_1$ and $\Gamma_2$. Hence, the piecewise definition of $\rho_t$ extends to all of $\Gamma$.

The fact that $c_t$ centralizes $\rho_0(S)$ ensures that $\rho_t$ is a well-defined representation. That is, the only ambiguity in the above piecewise definition is when $\gamma \in S$, but the two definitions agree on $S$.

Similarly, if $\Gamma= *_g (\Gamma')$ is an HNN extension, then we may define $\rho_t(\gamma)=\rho_0(\gamma)$ when $\gamma\in \Gamma'$ and $\rho_t(g)=c_t\rho_0(g)$ for the stable letter $g$.

When a manifold, $M$, has an embedded codimension-$1$ submanifold $\Sigma$ the fundamental group $\pi_1M$ decomposes either as a free product with amalgamation (when $\Sigma$ is separating) or as an HNN extension (when $\Sigma$ is non-separating). Given a representation $\rho_0$ of $\pi_1M$ into some Lie group we get a path of representations after choosing a path $c_t$ in $C(\rho_0(\pi_1\Sigma))$ by the previous construction. We refer to this path of representations as the \textit{bending of $\rho_0$ along $\Sigma$ with bending parameter $t$} after resolving any ambiguity regarding the path of centralizing elements. For the construction in this paper, the centralizer in question is one-dimensional, so no such ambiguity will arise (save reparametrizing this one-dimensional subgroup).

Note that both free products with amalgamation and HNN extensions are examples of graph of groups decompositions.
It is possible to discuss bending along many submanifolds simultaneously by describing the fundamental group as a graph of groups, and checking a cocycle condition on the intersections of the submanifolds \cite{Bart-Scannell}, \cite{Johnson-Millson}. However, we will be discussing bending deformations that are well-behaved with respect to one another, and hence we will not require a great deal of technology.

The classical example that engenders the name ``bending'' is that of a surface group representation into $\mathrm{Isom}(\mathbb{H}^2)$ bending into $\mathrm{Isom}(\mathbb{H}^3)$ along a simple closed curve. This gives a  family of non-trivial quasi-Fuchsian groups (for small elements of the centralizer). This example, as well as a detailed study of bending deformations and their relationship to graph of groups decompositions were first described by Thurston, and then later by Johnson and Millson \cite{Johnson-Millson}.

When bending a representation, the developing map may also be deformed in the following way. Fix a point $x_0\in \tilde{M}$. Define $dev_t(x_0)=dev_0(x_0)$. For any other point $x\in\tilde{M}$, choose a path $p:I\rightarrow \tilde{M}$ from $x_0$ to $x$. Generically, the pre-image of $\Sigma$ in $I$ is an oriented collection of points. The developing map of this path is adjusted by multiplying by an appropriate conjugate of $c_t$ at each point on $I$. The resulting developing map $dev_t$ is $\rho_t$-equivariant.

Let us suppose that we have chosen a codimension-1 submanifold in $M$ a manifold, some representation $\rho:\pi_1 M\rightarrow G$ where $G$ is some Lie group, and some non-trivial path $c_t$ in the centralizer $C(\rho_0(\pi_1\Sigma)))$. We use the notation $\rho_{(\Sigma,c_t)}$ or $\rho_{(\Sigma,t)}$ to mean the bent representation and $dev_{(\Sigma,c_t)}$ or $dev_{(\Sigma,t)}$ the associated developing map. When the parameter $t$ or the submanifold $\Sigma$ are clear from context, one or the other may be omitted from the subscript.

It is not true in all cases that arbitrary bending of the holonomy induces a geometrically nice structure on the associated developing map. We will show in Section \ref{section: main theorem} that the bending deformations we construct give convex projective structures using results from \cite{Cooper-Long-Tillmann}

\subsection{Iterated bending}\label{section: iterated bending} 

Suppose that $M$ is a manifold with two connected totally geodesic codimension-1 submanifolds $\Sigma_1$ and $\Sigma_2$, and that $\rho_0:\pi_1M \rightarrow G$ is a representation to some Lie group, $G$. Suppose also that we have chosen paths in the centralizer of $\rho_0(\pi_1\Sigma_i)$, parametrized as $c_t$ and $d_s$, respectively. Then we have the following.

\begin{lemma}\label{lemma: bending commutes}
If $c_t$ and $d_s$ commute, then $(\rho_{(\Sigma_1,t)})_{(\Sigma_2,s)}=(\rho_{(\Sigma_2,s)})_{(\Sigma_1,t)}$, and $(dev_{(\Sigma_1,t)})_{(\Sigma_2,s)}=(dev_{(\Sigma_2,s)})_{(\Sigma_1,t)}$. In this case, we will say that the two bending deformations commute.
\end{lemma}

For two arbitrary submanifolds, it is not clear that the second bending is defined: $d_s$ may not be in the centralizer of $\rho_{(\Sigma_1,t)}(\pi_1(\Sigma_2))$.

\begin{proof}
We address the case where both $\Sigma_1$ and $\Sigma_2$ are separating and intersect. The cases where one or both are non-separating or do not intersect follow similarly, with changes to the decomposition of the group $\Gamma$.

Let $A= \Sigma_1\cap\Sigma_2$, so that both $c_t$ and $d_s$ are paths in $C(\rho_0(\pi_1(A)))$. The key fact is that the first deformation does not change the centralizer of the second submanifold group. To see this, write $\pi_1(\Sigma_1)= H_1 *_{\pi_1(A)} H_2$, so that $\rho_{(\Sigma_2,s)}(\pi_1(\Sigma_1))$ is got by extending the piecewise definition:

$$\rho_{(\Sigma_2,s)}(\gamma))=
   \begin{cases} 
      \rho_0(\gamma) & \gamma\in H_1 \\
      d_s\rho_0(\gamma)d_s^{-1} & \gamma\in H_2 
   \end{cases}.
$$

The path $c_t$ commutes with every element from $\rho_0(H_1)$ and $\rho_0(H_2)$ as well as with $d_s$, and is hence in $C(\rho_{(\Sigma_2,s)}(\pi_1(\Sigma_1)))$ as claimed. Symmetrically, $d_s \in C(\rho_{(\Sigma_1,t)}(\pi_1(\Sigma_2)))$. This demonstrates that both iterated operations are defined. We need now to show that they are equal on all of $\pi_1(M)$.

\begin{figure}[h]
\centering
\includegraphics[width=.5 \textwidth]{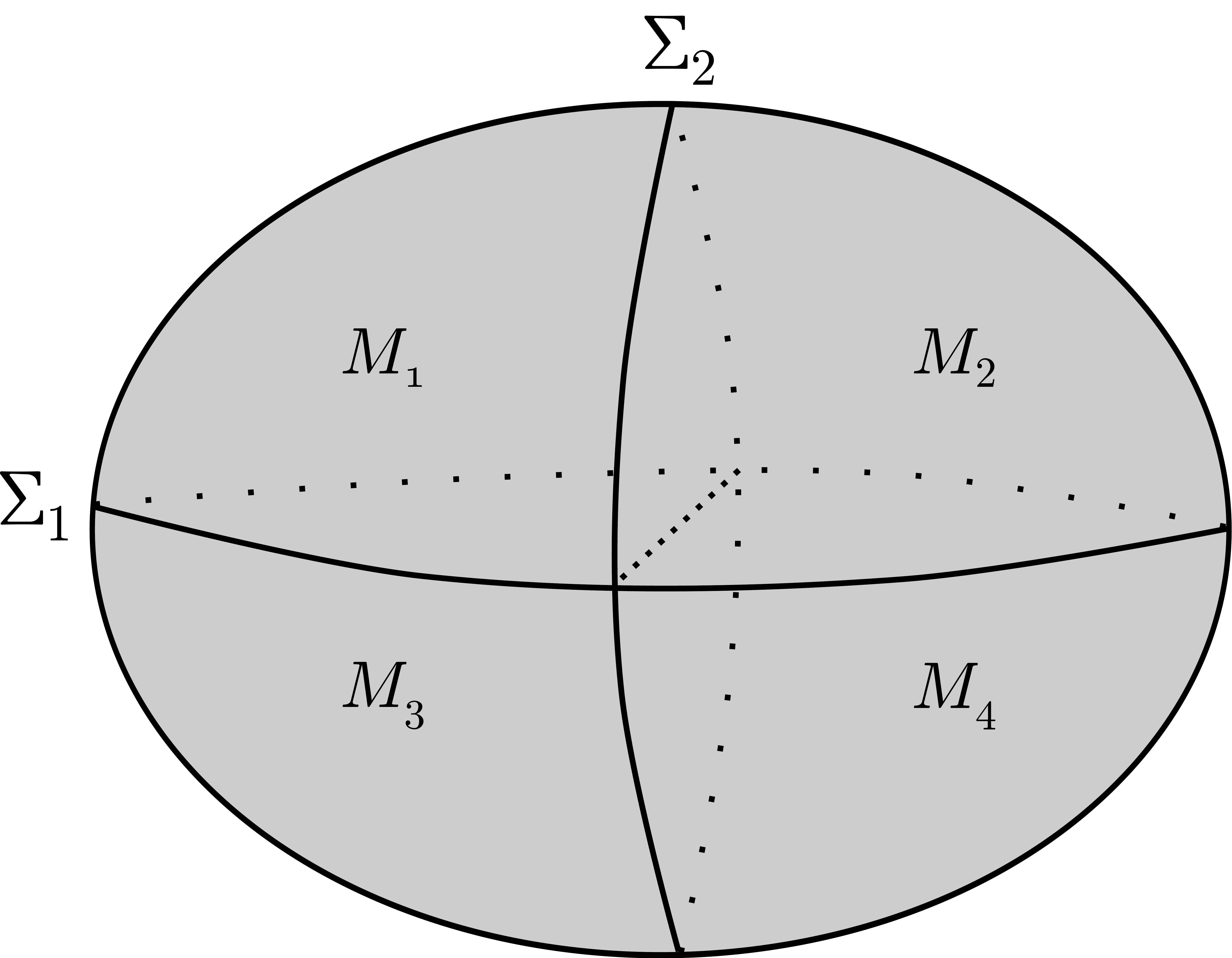}
\caption{Schematic of a manifold with two separating submanifolds.}
\end{figure}

Let $\Gamma = \pi_1(M)$. Under the assumption that $\Sigma_1$ and $\Sigma_2$ are separating and intersect, we see that $M \setminus (\Sigma_1\cup\Sigma_2)$ has four components, $M_1,\dots, M_4$. Choose the indexing so that $M\setminus \Sigma_1$ has components containing one each of $M_1\cup M_2$ and $M_3 \cup M_4$, while $M\setminus \Sigma_2$ has components containing one each of $M_1 \cup M_3$ and $M_2 \cup M_4$. Let $\Gamma_i=\pi_1(M_i)$, and note the decomposition of $\Sigma_i$ over $A$ by $\pi_1(\Sigma_1)=H_1*_{\pi_1(A)}H_2$ and $\pi_1(\Sigma_2)=K_1*_{\pi_1(A)}K_2$.

So, the group $\Gamma$ has two relevant decompositions: 
$$\Gamma= (\Gamma_1*_{K_1} \Gamma_2)*_{\pi_1(\Sigma_1)}(\Gamma_3*_{K_2} \Gamma_4)$$
and
$$\Gamma= (\Gamma_1*_{H_1} \Gamma_3)*_{\pi_1(\Sigma_2)}(\Gamma_2*_{H_2} \Gamma_4).$$

Proving that $(\rho_{(\Sigma_1,t)})_{(\Sigma_2,s)}=(\rho_{(\Sigma_2,s)})_{(\Sigma_1,t)}$ is only a matter of writing down the two pairs of piecewise definitions and noting their equality. Commutitivity is necessary, for if $\gamma\in\Gamma_4$, then $(\rho_{(\Sigma_2,s)})_{(\Sigma_1,t)}(\gamma)=c_t d_s \rho_0(\gamma) d_s^{-1} c_t^{-1}$ and $(\rho_{(\Sigma_1,t)})_{(\Sigma_2,s)}(\gamma) =d_s c_t \rho_0(\gamma) c_t^{-1} d_s^{-1}$ which are equal if $c_t$ and $d_s$ commute.
\end{proof}

Essentially, bending is just conjugating or multiplying certain elements of the group by elements in the centralizers. If the centralizing elements commute, it does not matter in which order they are multiplied.

The commutitivity condition required should not be considered generic. When two totally geodesic codimension-1 submanifolds $\Sigma_1$ and $\Sigma_2$ intersect generically and bending is performed along $\Sigma_1$, the other does not remain totally geodesic. Instead, it becomes piecewise totally geodesic with some bending locus $A$, a codimension-2 submanifold. The group $\pi_1(\Sigma_2)$ is decomposed as a free product over $\pi_1(A)$, and the restriction of the bending deformation of $M$ to $\Sigma_2$ is also a bending deformation (algebraically). Generically, for convex projective manifolds, this is a deformation into $\mathrm{PGL}(n+1,\R)$, but demanding commutativity of centralizers as described above guarantees that the bent representation of $\pi_1\Sigma_2$ remains in a copy of $\mathrm{PGL}(n,\R) < \mathrm{PGL}(n+1,\R)$.

With a proper understanding of graphs of groups decompositions, the above proof could probably be simplified considerably. However, a complete exposition on graphs of groups would be orthogonal to the purpose of this paper.

The previous lemma ensures that the following notion is well-defined.

\begin{definition}\label{def: iterated bending}
Let $M$ be a manifold and $\rho:\pi_1(M) \rightarrow G$ be a representation of its fundamental group into some Lie group. Suppose that $M$ has a collection of codimension-$1$ submanifolds, $\{\Sigma_i\}_{i=1}^k$, and that for each
 $\Sigma_i$ we choose a path $c_{s_i}$ in the centralizer $C(\rho(\pi_1(\Sigma_i))) < G$. Let 
$\mathcal{S}= \{(\Sigma_i,c_{s_i})\}_{i=1}^k$ be the set of pairs of submanifolds and paths. 
Given that the centralizing paths commute, we define $\rho_\mathcal{S}=\rho_{(\Sigma_1, c_{s_1})\dots(\Sigma_k,c_{s_k})}$ 
and refer to it as the \textit{bent representation} or the \textit{bending of} $\rho$ \textit{along} $\mathcal{S}$.

\end{definition}

\section{Bending convex projective domains}

\subsection{The paraboloid model of $\mathbb{H}^n$}\label{section: paraboloid model}

For algebraic analysis, we will use the following model for hyperbolic space. We will find $\mathbb{H}^n$ as the negative cone of a signature $(n,1)$ form on $\R^{n+1}$, but not the standard form. The advantage of this model is that one of the basis vectors for $\R^{n+1}$ lies on its boundary, lending a desirable form to the parabolic subgroup stabilizing it.

Let 

$$Q_n =  \left( \begin{array}{ccc}
0 & 0 & -1 \\
0 & I & 0 \\
-1 & 0 & 0 
\end{array} \right)$$
where $I$ represents an identity matrix of dimension $n-1$. The quadratic form $x \mapsto (^t x Q_n x)$ has signature $(n,1)$ on $\R^{n+1}$. 

Let us define $\mathbb{H}^n = \{[x]\in P(\R^{n+1}) \vert (^t x Q_n x) < 0\}$. The first basis vector, $e_1$ is in $\partial \mathbb{H}^n$, and the translational part of its stabilizer in the identity component of $\textrm{Isom}(\mathbb{H}^n)$ is upper triangular and of the form

$$H=\left( \begin{array}{ccc}
1 & ^tv & \sigma \\
0 & I & v \\
0 & 0 & 1 \\
\end{array} \right)$$
where $I$ is an identity matrix, $v = (v_2,\dots,v_n)$ is a vector of length $n-1$ and $\sigma = \frac{1}{2} \sum_{i=2}^n v_i^2$. Note that in the terminology from Section \ref{section: cusp neighborhoods} this group is exactly $H(\psi)$, when $\psi_i=0$ for all $i$.

\subsection{Centralizers of relevant groups}\label{section: centralizers of relevant groups}

We are interested in bending hyperbolic manifolds along codimension-$1$ geodesic submanifolds. Up to translation, the holonomy of the submanifold lies in a copy of $\SO(n-1,1)$ embedded reducibly in $\SO(n,1)$. In particular, consider the hyperplane $\Pi_2=\ker (e_2^*) \cap \mathbb{H}^n$. It is not difficult to show that the identity component of the centralizer of the copy of $\SO(n-1,1)$ in $\PGL(n+1,\R)$ stabilizing $\Pi_2$ (using the paraboloid model of $\mathbb{H}^n$) is

$$C_0(\SO(n-1,1)) = \left( \begin{array}{ccccc}
1 & 0 & 0 & 0 & 0 \\
0 & e^t & 0 & 0 & 0\\
0 & 0 & 1 & 0 & 0 \\
0 & 0 & 0 & \ddots & 0 \\
0 & 0 & 0 & 0 & 1 \\
\end{array} \right).$$

Centralizers of other $\Pi_i=\ker (e_i^*) \cap \mathbb{H}^n$ are similar, as they are conjugate, and the reader is directed to Ballas and Marquis for the proof \cite{Ballas-Marquis}.

\section{Analysis of bending cusp neighborhoods}\label{section: bending examples}
To understand the behavior of the geometry at a cusp under bending, we extend arguments from Ballas and Marquis \cite{Ballas-Marquis}

\subsection{Bending a hyperbolic cusp}\label{section: local bending}

Let us analyze an important, but highly non-generic example. We require the following technical definition about codimension-1 submanifolds in cusped hyperbolic manifolds.

\begin{definition}
A submanifold $\Sigma$ in a hyperbolic manifold $M$ with a cusp $P\cong B\times \R_{\geq 0}$ \textit{essentially intersects} $P$ if given any cusp subneighborhood $P_c = B\times [c,\infty)$, $P_c \cap \Sigma \neq \emptyset$.
\end{definition}

We will also say that $P$ and $\Sigma$ intersect \textit{essentially}, the idea being that $\Sigma$ ``goes out the cusp.''

Suppose that $M$ is an oriented hyperbolic $n$-manifold with a torus cusp $P \cong T^{n-1}\times \R_{\geq 0}$. Let us also suppose that $M$ has a collection of codimension-$1$ totally geodesic submanifolds and that paths have been chosen in their centralizers, so that $\mathcal{S}=\{(\Sigma_i, c_{s_i})\}_{i=2}^{n}$ has the following properties:
\begin{enumerate}
\item Each $\Sigma_i$ is embedded in $M$
\item All pairs $\Sigma_i$, $\Sigma_j$ are orthogonal
\item Each $\Sigma_i$ essentially intersects $P$
\item All pairs $\Sigma_i$, $\Sigma_j$ intersect non-trivially inside of $P$
\item For each $i$, the intersection $\Sigma_i \cap P$ has only one component.
\end{enumerate}

This is a non-generic arrangement of submanifolds, but we will guarantee that such example exists for each $n$ in section \ref{section: arithmetic manifolds}. It is an exercise in hyperbolic geometry to see that the orthogonality condition implies that centralizers of these submanifolds commute. Hence, iterated bending along this family is well-defined.

The final property ensures that each $\Sigma_i\in \mathcal{S}$ does not meet $P$ in a parallel collection of cusps. The second and fourth property together ensure that the intersection of all $\Sigma_i\in \mathcal{S}$ inside of $P$ is a single curve.

These observations allow us to choose useful coordinates on $\tilde{M}\cong \mathbb{H}^n$ (in the paraboloid model). In particular we may ensure that $P$ is covered by a horoball neighborhood of the point $[e_1]$, and that each $\Sigma_i\in\mathcal{S}$ is covered by a hyperplane of the form $\Pi_i=e_i^* \cap \mathbb{H}^n$, hence the unusual indexing of the submanifolds.

We now analyze the conjugacy class of the 
cusp group under iterated bending along $\mathcal{S}$.

\begin{lemma}\label{lemma: local bending}
Given the above arrangement, $dev_\mathcal{S}(P)$ is a type $t$ cusp if and only if exactly $t$ of the bending parameters $s_i$ are non-zero. Furthermore, as the parameters $\{s_i\}_{i=2}^n$ vary over $\R^{n-1}$, the cusp parameter $\psi = (\psi_2,\dots,\psi_n,0)$ where for $i>1$, $\psi_i$ is given by the formula $(\frac{b_i^2(e^{s_i}+1)}{2(e^{s_i}-1)(s_i)})$, where $b_i>0$ is a constant. 
\end{lemma}

\begin{proof}
Assume that the $s_i$ are indexed such that the first $t$ are non-zero, while the remaining are zero.

Let $P$ be centered at $e_1$ in the paraboloid model. In these coordinates, we have a generating set $\{\gamma_i\}_{i=2}^n$ for the cusp group, where
$$ \rho_0(\gamma_i)=
\left(\begin{array}{c c c}

1 & b_i(^t e_{i-1}) & \frac{b_i^2}{2}\\

0 & I & b_ie_{i-1} \\ 

0 & 0 & 1
\end{array}\right).$$

Here, $e_i$ is the $i$th basis vector of length $n-1$. The $b_i$ are some positive constants (determining the cusp shape). The fact that such a generating set exists is equivalent to saying that the cusp shape of $P$ is rectangular, which is guaranteed by the submanifolds in $\mathcal{S}$.

Note that $\gamma_i\in\pi_1(\Sigma_j)$ whenever $i\neq j$. Hence, bending along $\Sigma_j$ leaves  $\gamma_i$ fixed whenever $i \neq j$. In other words, for each $i$,

$$\rho_\mathcal{S}(\gamma_i)=\rho_{(\Sigma_i,s_i)}(\gamma_i)=\mathrm{diag}(1,\dots,1,e^{s_i},1,\dots,1)\rho_0(\gamma_i)$$

where the non-unit entry appears in the $i$th position.

We wish to determine which of the cusp domains described in section \ref{section: cusp neighborhoods} is stabilized by the group generated by $\{\rho_\mathcal{S}(\gamma_i)\}_{i=2}^n$. To do so, we will perform a change-of-basis by the matrix $A$, defined below. The motivation is that $\rho_\mathcal{S}(\gamma_i)$ ceases to be unipotent when $s_i\neq0$. That is, a new eigenvector appears in the plane spanned by $\{e_1,e_i\}$. To put the bent group in standard form, we must take this new eigenvector to $e_i$ while leaving $e_1$ fixed, and do the same in the dual (in the plane spanned by $\{e_1^*,e_i^*\}$). The matrix $A$ does exactly this simultaneously for each $i$. 

Conjugate $\rho_\mathcal{S}(\gamma_i)$ by the matrix

$$A= 
\left(\begin{array}{c c c c}

1 & ^t (\frac{-b_j}{e^{s_j}-1})_{j=2}^{t+1} & 0 & 0\\

0 & I & 0 & (\frac{-b_j}{e^{s_j}-1}e^{s_j})_{j=2}^{t+1}\\ 

0 & 0 & I & 0\\

0 & 0 & 0 & 1
\end{array}\right).$$

For $i=2,\dots, t$, this results in

$$\gamma_i'=A\rho_\mathcal{S}(\gamma_i) A^{-1} =
\left(\begin{array}{c c c c c}

1 & 0 & 0 & 0 & \frac{-b_i^2(e^{s_i}+1)}{2(e^{s_i}-1)}\\

0 & I & 0 & 0 & 0 \\

0 & 0 & e^{s_i} & 0 & 0 \\ 

0 & 0 & 0 & I & 0\\

0 & 0 & 0 & 0 & 1
\end{array}\right)$$

where the non-zero diagonal entry appears in the $(i,i)$ position. The remaining $\gamma_j$ are fixed by this conjugation, so for $j=t+1,\dots,n$ let $\gamma_j'=\rho_0(\gamma_j)$.

Note that the sets $P_c$ in $P(\R^{n+1})$ given by

$$P_c = \left\{\left[\left(c-\sum_{i=2}^{t+1} a_i\log(x_i) + \frac{1}{2}\sum_{i=t+2}^{n} x_i^2\right),x_2,\dots,x_{n},1\right]\right\}$$

with 

$$a_i=\frac{b_i^2(e^{s_i}+1)}{2(e^{s_i}-1)(s_i)}$$

are preserved by all $\gamma_i'$. Hence, $\rho_\mathcal{S}(\pi_1 P)$ is a subgroup of $H(\psi)$, and by Theorem 0.2 of \cite{Ballas-Cooper-Leitner}, $dev_\mathcal{S}(P)$ is equivalent to a type-$t$ cusp corresponding to $\psi=(a_1,\dots,a_t,0,\dots,0)$ (after reordering the first $t$ coordinates so that $a_i$ are non-increasing).

\end{proof}

\begin{remark}\label{remark: goofy topology}
It should be noted that the map $\R_{\geq0}^{n-1} \rightarrow \mathscr{W}^n$ which sends 
$$(s_1,\dots,s_n) \mapsto \psi=(a_1,\dots a_k,0,\dots,0)$$
is continuous, except when the type changes. This discontinuity can be rectified by instead taking 
$$(s_1,\dots,s_n) \mapsto (a_1^{-1},\dots a_k^{-1},0,\dots,0),$$
suggesting that the type parameters for generalized cusps should be inverted. This is known to Ballas, Cooper, and Leitner, and should be reflected in forthcoming literature.
\end{remark}

\section{Examples from Arithmetic Manifolds}\label{section: arithmetic manifolds}

The previous sections shows that cusp type may be varied by bending along a family of submanifolds with special intersection requirements. However, it is not obvious that there exist complete finite-volume hyperbolic $n$-manifolds that have the appropriate submanifolds to perform the iterated bending described. Demonstrating the existence of such examples is the purpose of this section. 

The main tool we will use is the separability properties of arithmetic hyperbolic manifolds.  A general reference for arithmetic hyperbolic lattices is Morris' book \cite{Morris}. For the reader wholly unfamiliar with arithmetic lattices and manifolds, it will suffice to know that these are a subclass of cusped and compact hyperbolic manifolds appearing in every dimension. Let us recall the following theorem from \cite{Agol-Long-Reid}, \cite{Bergeron-Haglund-Wise}.

\begin{theorem}[Corollary 1.12, Bergeron-Haglund-Wise \cite{Bergeron-Haglund-Wise}]\label{theorem: GFERF}
Arithmetic hyperbolic orbifold and manifold groups of simplest type and with finite covolume are separable on geometrically finite subgroups.
\end{theorem}

This property has the name ``geometrically finite extended residually finite'' (GFERF). It will be the key to constructing examples of hyperbolic manifolds where we may apply iterated bending to achieve deformed cusps.

We will require two additional separability properties in order to precisely control behavior of submanifolds at cusps.

\begin{theorem}[Theorem 1.3, McReynolds \cite{McReynolds}]\label{theorem: PERF}
If $\Gamma < \mathrm{Isom}(\mathbb{H}^n)$ is an arithmetic lattice and $v \in \partial \mathbb{H}^n$, then $\mathrm{stab}(v)\cap \Gamma$ is separable in $\Gamma$.
\end{theorem}

This will allow us to ensure a submanifold does not intersect a cusp more than once. The following is folk-lore, but was proved in \cite{McReynolds-Reid-Stover}.

\begin{theorem}[Proposition 3.1, McReynolds-Reid-Stover \cite{McReynolds-Reid-Stover}]\label{theorem: torus cusps}
Every finite volume hyperbolic $n$-manifold $M$ has a finite cover $M'$ so that $M'$ has only torus cusps and has at least two cusps.
\end{theorem}

We proceed with the several lemmas to be applied sequentially in Theorem \ref{theorem: existence}.

\begin{lemma}[Selberg's lemma]\label{lemma: Selberg}
A finitely generated linear group over a field of characteristic $0$ is virtually torsion free.
\end{lemma}

We refer the reader to \cite{Nica} for a proof of Selberg's lemma. In terms of covering space theory, this means that an orbifold whose fundamental group can be realized in a linear group (for example a hyperbolic orbifold) has a finite-sheeted manifold cover. 

The following is a technical definition that will ease in the statement of the following lemmas.

\begin{definition}[Property $(\star)$]\label{def: star}
Suppose a hyperbolic $n$-manifold (orbifold) has $n-1$ codimension-1 totally geodesic immersed submanifolds (suborbifolds), $\mathcal{S}=\{\Sigma_i\}_{i=1}^{n-1}$, and a cusp neighborhood $P$. We will say that $\mathcal{S}$ has \textit{property $(\star)$ at $P$} if each $\Sigma_i\in \mathcal{S}$ intersects $P$ essentially, and all $\Sigma_i$ share a common intersection which is a ray inside of $P$, and at this intersection they are pairwise orthogonal. 
\end{definition}

The intent is that the submanifolds should be orthogonal inside of $P$, but no pair is allowed to be vacuously orthogonal (parallel in $P$).

The following lemma describes how we may lift immersed submanifolds to embedded submanifolds, while retaining a desired intersection.

\begin{lemma}[Lifting immersed submanifolds to embedded]\label{lemma: immersed to embedded}
Let $\mathcal{S}=\{\Sigma_i\}_{i=1}^{n-1}$ be a family of immersed totally geodesic submanifolds in a finite volume cusped arithmetic hyperbolic manifold $M$. Suppose that there is a cusp neighborhood $P \subset M$ so that $\mathcal{S}$ has property $(\star)$ at $P$.

Then there is a finite-sheeted cover $M'$ of $M$ with a cusp $P'$ covering $P$ and for each $i$ an embedded totally geodesic submanifold $\Sigma_i'$ covering $\Sigma_i$ so that the collection $\{\Sigma_i'\}_{i=1}^{n-1}$ has property $(\star)$ at $P'$.
\end{lemma}

It is well-known that subgroup separability allows lifting of immersed submanifolds to embedded submanifolds in a finite sheeted cover. The trick is to maintain the intersection properties at a cusp. 

\begin{proof}

We proceed inductively. Suppose that for $i=1,\dots j-1$ we have that $\Sigma_i$ is embedded in $M$ (and the intersection properties are as described in the lemma). $\Sigma_j$ may be immersed.

\begin{figure}[h]
\centering
\includegraphics[width=.5 \textwidth]{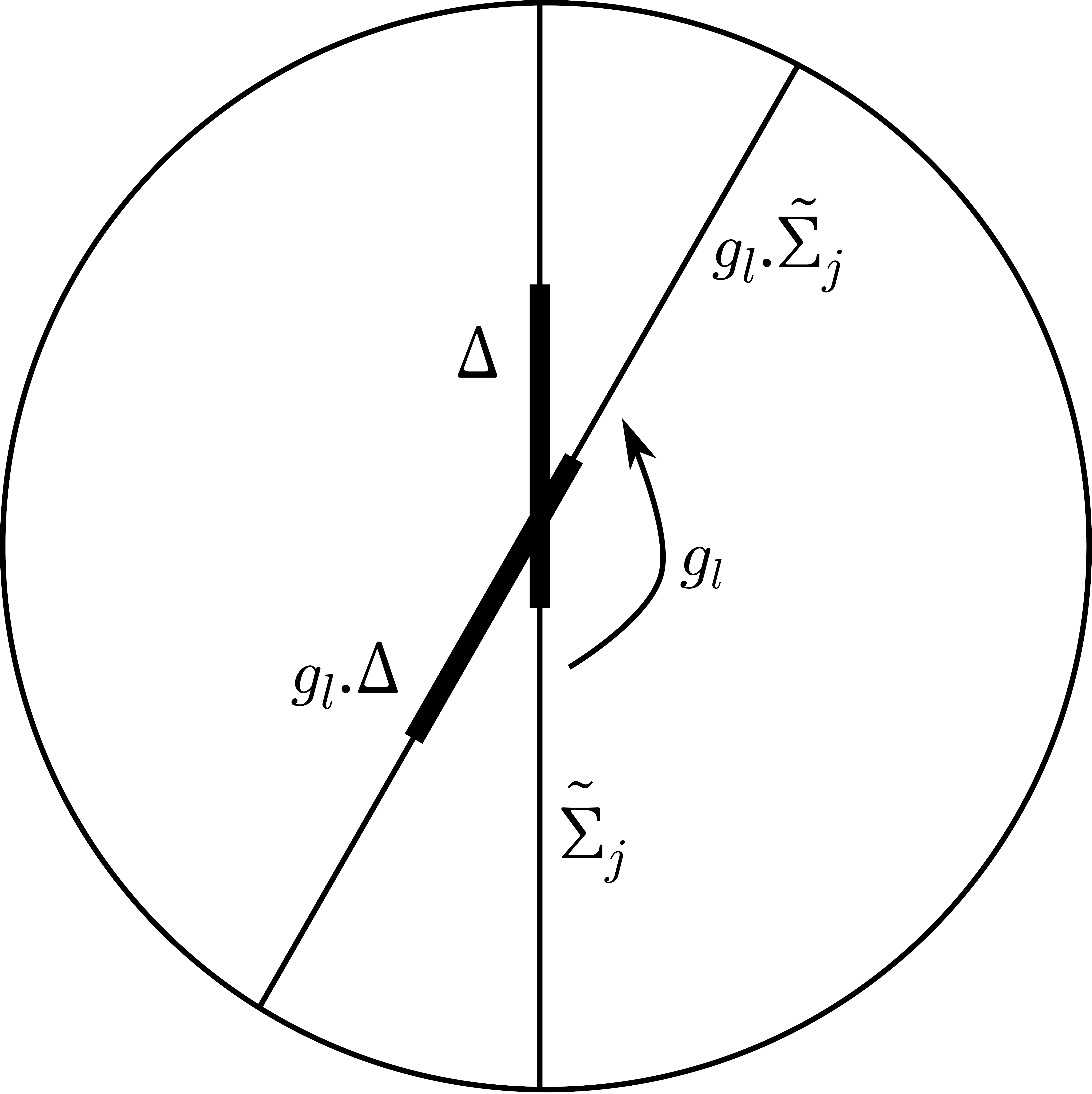}
\caption{An element $g_l \in \pi_1 M$ causes an unwanted self-intersection of a submanifold.}
\end{figure}

Choose a copy of the universal cover of $\Sigma_j$, $\tilde{\Sigma}_j$, in the universal cover of $M$ and a fixed fundamental domain $\Delta$ for the action of $\pi_1(\Sigma_j)$. There is some finite (possibly empty) collection of deck transformations $\{g_l\}\subset \pi_1(M) \setminus \pi_1(\Sigma_j)$ so that $g_l.\Delta\cap\Delta\neq\emptyset$. All other intersections of copies of $\tilde{\Sigma}_j$ are $\pi_1(\Sigma_j)$-translates of those in $\Delta$.

Since $\Sigma_j$ is totally geodesic, $\pi_1\Sigma_j$ is geometrically finite. Applying Theorem \ref{theorem: GFERF} gives a subgroup $H<\pi_1(M)$ with finite index so that $\pi_1(\Sigma_j)<H$ and $g_l\notin H$ for all $l$. Let $M_H$ be the corresponding finite-sheeted cover of $M$. In this cover, $\Sigma_j$ has a lift $\Sigma_j'$ that is embedded.

It remains to show that we can pick covers of $\Sigma_i$, $i<j$ that retain the desired intersection properties. If we do not make a careful choice, it is possible that the covers intersect in no cusp at all. $P$ is covered by some finite number of cusps in $M_H$. Since $\mathcal{S}$ has property $(\star)$ in $P$, we may pick a point in their mutual intersection and examine an evenly covered neighborhood $U$ of that point. Choosing $U$ to be sufficiently small, we have that the submanifolds intersect locally in the canonical arrangement of orthogonal $(n-1)$-planes in $\mathbb{H}^n$. The neighborhoods covering $U$ are isometric to $U$ by restrictions of the covering map. Choose a covering neighborhood $\tilde{U}$ in $M_H$ which $\Sigma_j'$ intersects, and choose $\{\Sigma_i'\}_{i=1}^{n-1}$ to be the covers of $\Sigma_i$ which also intersects $\tilde{U}$. Now for $i=1,\dots,j$, $\Sigma_i'$ is embedded and the collection $\{\Sigma_i'\}_{i=1}^{n-1}$ intersect in $P'$ as desired.

Repeating this process results in a finite sheeted-cover of $M$ with $n-1$ embedded submanifolds that retain property $(\star)$ at some cusp.
\end{proof}

The technique used above, wherein we use an evenly-covered neighborhood of the desired intersection is the key. We use a very similar argument to eliminate non-orthongal mutual intersections of the submanifolds away from the cusp in the following lemma.

\begin{lemma}[Eliminating non-orthogonal mutual intersections]\label{lemma: remove bad intersections}
Let $\mathcal{S}=\{\Sigma_i\}_{i=1}^{n-1}$ be embedded totally geodesic submanifolds in a finite volume cusped arithmetic hyperbolic manifold $M$. Suppose that there is a cusp neighborhood $P \subset M$ so that $\{\Sigma_i\}$ have property $(\star)$ at $P$. 

Then there is a finite-sheeted cover $M'$ of $M$ and covers $\Sigma_i'$ of each $\Sigma_i$ so that $\{\Sigma_i'\}_{i=1}^k$ meet only pairwise orthogonally and  $\{\Sigma_i'\}$ have property $(\star)$ at some cusp $P'$ covering $P$.
\end{lemma}

\begin{figure}[h]
\centering
\includegraphics[width=.5 \textwidth]{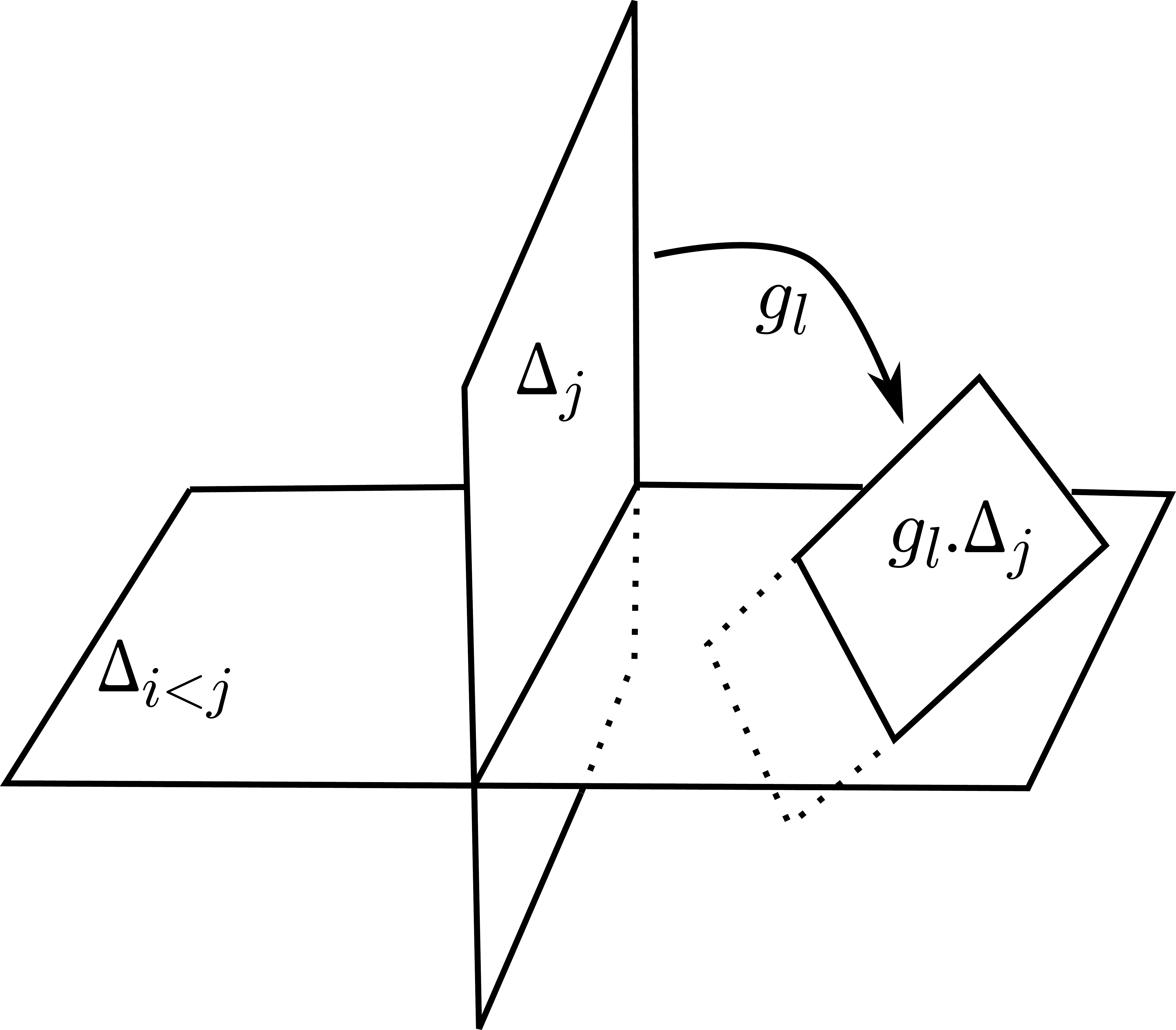}
\caption{An element $g_l \in \pi_1 M$ causes two submanifolds to intersect non-orthogonally.}
\end{figure}

\begin{proof}
The argument is similar to the previous lemma. Suppose for induction that $\Sigma_1,\dots,\Sigma_{j-1}$ intersect only orthogonally. For each $i=1,\dots,j$, choose a fundamental domain, $\Delta_i$ for the action of $\pi_1(\Sigma_i)$ on some copy of its universal cover, $\tilde{\Sigma_i}$ in $\mathbb{H}^n\cong\tilde{M}$, and make this choice so that all copies intersect orthogonally. There is some finite set $\{g_l\}_{l\in I}\subset \pi_1(M)$ ($I$ finite) so that $g_l.\Delta_j$ intersects some other $\Delta_i$ non-orthogonally.

Separability provides a finite-index subgroup, $H$, including $\pi_1(\Sigma_j)$ but excluding all $g_l$. Let $M_H$ be the corresponding covering space, and lift $\Sigma_j$ to $\tilde{\Sigma}_j$ in $M_H$. Choose an evenly covered neighborhood, $U$ of the intersection of $\mathcal{S}$ in $P$, and a cover $\tilde{U}\subset \tilde{P}$ which intersects $\tilde{\Sigma}_j$. $\tilde{U}$ is isometric to $U$, and we may use it as before to choose finite-sheeted covers for $\Sigma_i \in \mathcal{S} \setminus \{\Sigma_j\}$ which have property $(\star)$ at $\tilde{P}$. 

Proceeding inductively proves the lemma.
\end{proof}

The next lemma will be used to ensure that each submanifold meets a cusp $P$ in only one component

\begin{lemma}[Ensuring unique cusp intersections]\label{lemma: single intersections at cusp}
Let $\mathcal{S}=\{\Sigma_i\}_{i=1}^{n-1}$, be embedded totally geodesic submanifolds in a finite-volume arithmetic hyperbolic manifold $M$. Suppose that the family $\{\Sigma_i\}$ intersect pairwise orthogonally and has property $(\star)$ at some cusp $P\cong T\times \R_{\geq0}\subset M$ (where $T$ is the $(n-1)$-torus).

Then there is a finite cover $M'$ of $M$ containing covers $\{\Sigma_i'\}_{i=1}^k$ of $\Sigma_i$, and some cover $P'\cong T'\times \R_{\geq0}$ of $N$ with all of the above intersection properties and the additional property that $\Sigma_i'\cap P'$ has exactly one component.
\end{lemma}

\begin{proof}
Suppose that for $i=1,\dots,j-1$, $\Sigma_i\cap P$ has exactly one component. Choose a cusp cross-section, $T$,  so that for each $i$, the intersection of $\Sigma_i$ with $T\times \R \subset P$ is a collection of cusps in $\Sigma_i$. Choose $\Delta_j$ and $\Delta_T$ intersecting fundamental domains for the actions of $\pi_1(\Sigma_j)$ and $\pi_1(T)$ on fixed copies of $\tilde{\Sigma}_j$ and $\widetilde{T\times\{0\}}$, respectively.

There is some finite collection of $g_l \in \pi_1(M) \setminus \pi_1(\Sigma_j)$ so that $g_l.\Delta_j\cap\Delta_T\neq\emptyset$. Using Theorem \ref{theorem: PERF}, we have a finite index subgroup $H<\pi_1(M)$ so that $\pi_1(P)<H$ but $g_l\notin H$. Let the corresponding finite-sheeted cover be $M'$. Choose $P'$ a lift of $P$ and $\Sigma_j'$ a finite cover of $\Sigma_j$ which intersects $P'$.

As in the previous two lemmas, we choose a cover $\Sigma_i'$ for each $\Sigma_i$, $i = 1,\dots,j-1$ using an evenly covered neighborhood in $P$. Proceeding inductively, we may find a finite cover of $M$ and covers of $\Sigma_j$ where all covers intersect a cusp covering $P$ exactly one time.

\end{proof}

With these lemmas in hand we are prepared to create examples of hyperbolic $n$-manifolds to which we may apply the iterated bending procedure from Section \ref{section: bending examples}.

\begin{theorem}\label{theorem: existence}
There exists a finite volume hyperbolic manifold $M^n$ in every dimension that has the following properties

\begin{enumerate}
    \item There are $(n-1)$ totally geodesic codimension-$1$ embedded submanifolds $\mathcal{S}=\{\Sigma_i\}_{i=1}^{n-1}$ in $M^n$.
    \item The family $\mathcal{S}$ intersect pairwise orthogonally.
    \item There is a torus cusp neighborhood $P\subset M$ such that the family $\mathcal{S}$ has property $(\star)$ at $P$.
    \item For each $i$, the intersection $P\cap \Sigma_i$ has exactly one component.
\end{enumerate}
\end{theorem}

Inspiration comes from the lattice $\mathrm{SL}(2,\Z[i])<\mathrm{SL}(2,\mathbb{C})$ acting on $\mathbb{H}^3$ (in the upper half-space model). The two subgroups $\mathrm{SL}(2,\Z)$ and its conjugate under $i$-multiplication are lattices stabilizing hyperplanes. The quotient is an orbifold, and has a cusp centered at $0\in \partial \mathbb{H}^3$ which both codimension-1 suborbifolds intersect orthogonally.

\begin{proof}

Consider the quadratic form $j_n= x_1^2+\dots+x_n^2-x_{n+1}^2$ and the associated hyperbolic space $\mathbb{H}^n= \{v\in\R^{n+1} \vert  j_n(v,v)=-1, v_{n+1}>0\}$. Let the lattice $\Gamma = \mathrm{SO}(j_n,\Z)$. The group $\Gamma$ is an arithmetic lattice of simplest type, and the quotient $\Gamma\backslash \mathbb{H}^n$ is a cusped hyperbolic orbifold of finite volume.

There are $n$ obvious sublattices stabilizing hyperplanes, corresponding to the obvious embeddings of $\mathrm{SO}(j_{n-1},\Z)$ as hyperplane stabilizers (each inducing a reducible representation $\SO(n-1,1)$ into $\SO(n,1)$). Call these integer matrix groups $\Lambda_i$ (these are the subgroups of $\Gamma$ where the $i^{th}$ row and column of the matrices are zero except the diagonal entry, which is one).

Note that each of these subgroups is geometrically finite. The quotient $M_0 =\Gamma \backslash \mathbb{H}^n$ is a hyperbolic orbifold of finite volume with cusps, and there are $n$ geometrically finite sub-orbifolds given by quotients of hyperplanes by $\Lambda_i$. Any $(n-1)$ of these suborbifolds meet with property $(\star)$ at some cusp. If we choose $\{\Lambda_i\}_{i=1}^{(n-1)}$, then the cusp $P_0$ corresponding to the point $[0:0:\dots:1:1]\in\partial\mathbb{H}^n$ is one such cusp (here we denote points in $\partial\mathbb{H}^n$ by a projective line). Call the family of suborbifolds intersecting at $P_0$ with property ($\star$) $\mathcal{S}=\{\Sigma_i\}_{i=1}^{n-1}$.

Because $\Gamma$ is linear, Selberg's lemma \ref{lemma: Selberg} ensures a finite index subgroup $\Gamma'$ which is torsion free. We may also take a further finite-index subgroup to ensure that all cusps are torus cusps by Theorem \ref{theorem: torus cusps}. Let $M_1=\Gamma' \backslash \mathbb{H}^n$ be the corresponding manifold covering $M_0$. Choose $P_1$ to be some cusp in $M_1$ covering $P_0$. There are some submanifold $\mathcal{S}_1=\{\Sigma_{i,1}\}_{i=1}^{n-1}$ which are covers of $\Sigma_i$ that intersect $P_1$ essentially. Certainly, $\mathcal{S}_1$ can be chosen to intersect with property ($\star$) at $P_1$.

It is very possible that some or all $\Sigma_{i,1}\in\mathcal{S}_1$ are immersed, so we apply Lemma \ref{lemma: immersed to embedded} to find $M_2$ with $\mathcal{S}=\{\Sigma_{i,2}\}_{i=1}^{n-1}$ embedded and intersecting some cusp $P_2$ with property $(\star)$. Similarly, we may now apply Lemma \ref{lemma: remove bad intersections} to find another cover $M_3$ with embedded submanifolds $\mathcal{S}_3=\{\Sigma_{i,3}\}_{i=1}^{n-1}$ intersecting with property $(\star)$ at a cusp $P_3$ but with no non-orthogonal intersections. Lastly, Lemma \ref{lemma: single intersections at cusp} ensures that there is a further cover $M_4$ with submanifolds $\mathcal{S}_4=\{\Sigma_{i,4}\}_{i=1}^{n-1}$ and a cusp $P_4$ so that the submanifolds retain property $(\star)$ at $P_4$, and so that for each $i$, $\Sigma_{i,4}\cap P_4$ is a single component. This completes the proof, $M_4$ being the desired manifold in dimension $n$, with submanifold $\mathcal{S}_4$ the intersecting family of submanifolds, and $P_4$ the distinguished cusp.

\end{proof}

The examples which are guaranteed in the above theorem are those on which we will perform the iterated bending procedure from Section \ref{section: local bending}. 

\section{Producing Examples of Non-Standard Cusps in Arithmetic Manifolds}\label{section: main theorem}

We now have nearly everything in place to prove the main theorem. The strategy will be to take the manifolds provided by Theorem \ref{theorem: existence} and bend along the family of orthogonal submanifolds. However, we must guarantee that the resulting representation is the holonomy of some convex projective structure.

For this, we require the following, which is a special case of a theorem of Cooper, Long, and Tillmann. Let $M$ be a convex projective manifold with finitely many ends, all of which are generalized cusps. Let $VFG(M)$ be the set of representations of its fundamental group so that the cusp subgroups have images virtually conjugate into the group of upper-triangular matrices, and let $Rep_{ce}(M)$ be those representations giving a convex projective structure on $M$.

\begin{theorem}[Openness of convex projective structure, \cite{Cooper-Long-Tillmann}]\label{theorem: upper triangular is great}\
The set $Rep_{ce}(M)$ is an open subset of $VFG(M)$.
\end{theorem}

For the purposes of this paper, this means that if we can guarantee that all cusps of $M$ have upper triangular holonomy after iterated bending, then for small bending parameters, iterated bending gives the holonomy of a convex projective structure on $M$. We will show that this condition is true in the course of the following, the main theorem.

\begin{theorem}[Main Theorem]\label{theorem: main theorem}

For each $n > 1$ and each $\psi \in \partial\mathscr{W}^n$, there exists a connected convex projective manifold $M$ (with no boundary) of dimension $n$ which decomposes as the union of a compact manifold with boundary and a finite collection of generalized cusps, one of which has cusp parameter $\psi$ (and is hence of type $t(\psi)$).
\end{theorem}

\begin{proof}

Let $M$ be the $n$-dimensional manifold given by Theorem \ref{theorem: existence}, and choose $(s_i)_{i=1}^{n-1}$ some bending parameters. Then let $\mathcal{S}=\{(\Sigma_i,s_i)\}_{i=1}^{n-1}$ be the collection of codimension-1 submanifolds with the desired intersection properties at $P$ guaranteed by Theorem \ref{theorem: existence} equipped with the chosen bending parameters. Let $\rho_\mathcal{S}$ be the representation obtained by iterated bending along $\mathcal{S}$ of the hyperbolic holonomy of $M$. 

The analysis in Lemma \ref{lemma: local bending} applies to $\rho_\mathcal{S}(\pi_1P)$, because the submanifolds intersect at some cusp orthogonally with property ($\star$), and only meet the cusp once. By varying the bending parameters, $P$ may be made to have parameter $\psi$, for any $\psi\in \partial\mathscr{W}^n$. Recall from Theorem \ref{theorem: B-C-L} that cusp parameters $\psi$ and $\psi'$ are equivalent if they differ by a positive scalar. Hence, any cusp parameter $\psi$ may be achieved (up to equivalence) by using arbitrarily small bending parameters. That is, given any $\epsilon>0$ and $\psi \in \partial\mathscr{W}^n$, we may scale $\psi$  to $k\psi$, so that for each $i$ where $\psi_i$ is non-zero, 

$$k \psi_i> \frac{b_i^2(e^{\epsilon}+1)}{2(e^{\epsilon}-1)(\epsilon)},$$

the function for the cusp parameters found in Section \ref{section: bending examples}.

All that remains is to show that for sufficiently short bending parameters, the holonomies of the other cusps of $M$ are conjugate to upper triangular groups. To see this, suppose $P'$ is some other cusp of $M$. Change coordinates so that $P'$ is centered at $e_1$, and its hyperbolic holonomy is in the type-0 cusp group (the standard parabolic subgroup). Some subset $\mathcal{S}'$ of $\mathcal{S}$ meet $P'$. Since these surfaces are orthogonal, they meet $P'$ in a collection of orthogonal or parallel codimension-1 submanifolds. Perform a change of coordinates fixing $e_1$ and rotating in the hyperplane $\ker(e_1^*)$, composed with some parabolic element in $H(0)$ (the cusp group) so that each of these hyerplanes is $\ker(e_i^*)$ (for some $i$ not equal to $1$) or some parabolic translate of $\ker(e_i^*)$. 

Note that this change of coordinates preserves the parabolic subgroup. Furthermore, note that bending along these submanifolds now amounts to multiplication by diagonal elements and their conjugates finitely many ends, all of which are generalized cusps. Let $VFG(M)$ be the set of representations of its fundamental group so that by upper-triangular parabolics. Hence, the bent group remains upper triangular. Now, Theorem \ref{theorem: upper triangular is great} completes the proof by providing a small neighborhood of convex projective holomies achieved by the iterated bending around the initial hyperbolic holonomy of $M$. That is, there exists some $\epsilon>0$ so that when (for all $i$) $s_i<\epsilon$, $\rho_\mathcal{S}$ is the holonomy of a convex projective structure on $M$, and all of the cusps of $M$ are generalized cusps.

The decomposition of $M$ into a union of compact core and non-compact ends is topological data, and the bending deformation of the geometric structure on $M$ does not change its topology.

\end{proof}

Note that as $\epsilon$ approaches zero, the quantity in the proof above tends to $\infty$. As stated in Remark \ref{remark: goofy topology}, this is a result of the fact that the map taking a generalized cusp to its parameter $\psi$ is not continous with respect to the usual topology on $\mathscr{W}^n$.

\section{Further questions}\label{section: further questions}
While we have demonstrated the existence of interesting manifolds possessing almost all non-standard cusp types, some important questions remain. Theorem \ref{theorem: existence} utilizes separability arguments so frequently that it seems likely that the resulting manifold with the desired submanifold arrangement is a large cover of the original orbifold. It would be difficult to identifiy it more concretely using this construction, and it may be of interest to find more tractable manifolds with the desired submanifold arrangement.

The argument presented here gives control over the geometry at a single cusp. It would be a desirable extension to gain control over more or all ends of the manifold $M$, although it is not clear exactly how this would be done. A worthy goal would be to produce examples of manifolds with at least $k(t)$ cusps of type $t$ for all $t<n$, for any numbers $k(t)$ desired. The author does not think that this is a trivial extension, however.

In the course of this paper, a previous argument suggested that, in fact, bending along orthogonal submanifolds for arbitrary time parameters should give convex projective structures. This is known in the case of bending hyperbolic manifolds due to Marquis \cite{Marquis}. This argument was not used, as it is unnecessary and quite long. Furthermore, it should be the case that for a bending path $\rho_\mathcal{S}$, the map $h_\mathcal{S}=dev_0^{-1}\circ dev_\mathcal{S}:\Omega_0\rightarrow \Omega_S$ should extend to a homeomorphism $\bar{h}_\mathcal{S}$ on the complement of parabolic fixed points of $\partial\Omega_0$. In addition, the author believes it is possible to show using careful analysis on the boundary that the only segments appearing in the boundary of $\Omega_S$ are part of cusp boundaries.

\subsection{Diagonalizable cusps}

Type $n$ (diagonalizable) cusps are not achieved in this paper. Yet they ought to be the most common, as any coincidence of eigenvectors should be non-generic. Diagonalizable cusps also offer the possibility of constructing projective manifolds which are not deformations of hyperbolic manifolds. Benoist  explores the role of properly embedded triangles in the context of $3$-dimensional convex domains with cocompact actions by discrete groups (``divisible'' convex sets), and shows that the resulting quotient manifolds geometrically JSJ decompose into hyperbolic pieces along the quotients of these triangles \cite{Benoist}. The JSJ pieces achieved this way are manifolds with type $n$ cusps.

The technology of Coxeter polytopes provides examples of convex projective manifolds with diagonal cusps in small dimensions. In dimension not greater than $6$, Benoist uses this method to construct compact convex projective manifolds with properly embedded codimension-$1$ simplices. These examples decompose into unions of manifolds with ends that are all type $n$ cusps \cite{Benoist}.

There is some possibility of bending an $n$-manifold with a type $(n-1)$ cusp along a codimension-$1$ submanifold to achieve a manifold with a type $n$ cusp. However, finding such a submanifold that is embedded and whose fundamental group has non-trivial centralizer seems difficult. We can observe directly that such a deformation is possible in the model type $(n-1)$ cusp. Below is an outline of this calculation. Recall the model domains and groups from Section \ref{section: cusp neighborhoods}.

Note that the intersection of the hyperplane $e_1^*$ with the $n$-dimensional type $(n-1)$ model domain is a type $(n-1)$ cusp comain in dimension $(n-1)$. Its stabilizer is a diagonal group, given below as $H$, where $\boldsymbol{D}$ denotes a diagonal matrix.

$$H = \left\{\left( \begin{array}{ccc}
1 & 0 & 0  \\
0 & \boldsymbol{D} & 0 \\
0 & 0 & 1 
\end{array} \right)\right\}$$

The centralizer $Z(H)$ contains a diagonalizable subgroup which fixes every point in the hyperplane $e_1^*$ and the point $(e_1+k e_{n+1})$, namely 

$$Z'(H)=\left\{\left( \begin{array}{ccc}
\lambda & 0 & 0 \\
0 & \boldsymbol{I} & 0\\
k(\lambda-1) & 0 & 1   
\end{array} \right)\right\}.$$

It is an algebraic exercise to see that bending along this submanifold with a path in the centralizer parametrized by $\lambda$ produces a diagonalizable group. 

Suppose $M$ is an $n$-manifold with a type $(n-1)$ cusp $P$ and presume we choose coordinates for the domain covering $M$ so that a domain covering $P$ is coordinatized as the model type $(n-1)$ domain. If we can find a codimension-$1$ embedded submanifold meeting the boundary of the cusp in the simplex spanned by $e_2,\dots,e_n$, and so that the fundamental group of this submanifold acts reducibly with a preserved one-dimensional subspace $(e_1+k e_{n+1})$, we may bend along this submanifold to produce a representation where the holonomy of some cusp of $M$ deforms to type $n$. Theorem \ref{theorem: upper triangular is great} then guarantees that at least for some small values of $\lambda$, $M$ retains a convex projective structure.

\begin{figure}[h]
\centering
\includegraphics[width=.5 \textwidth]{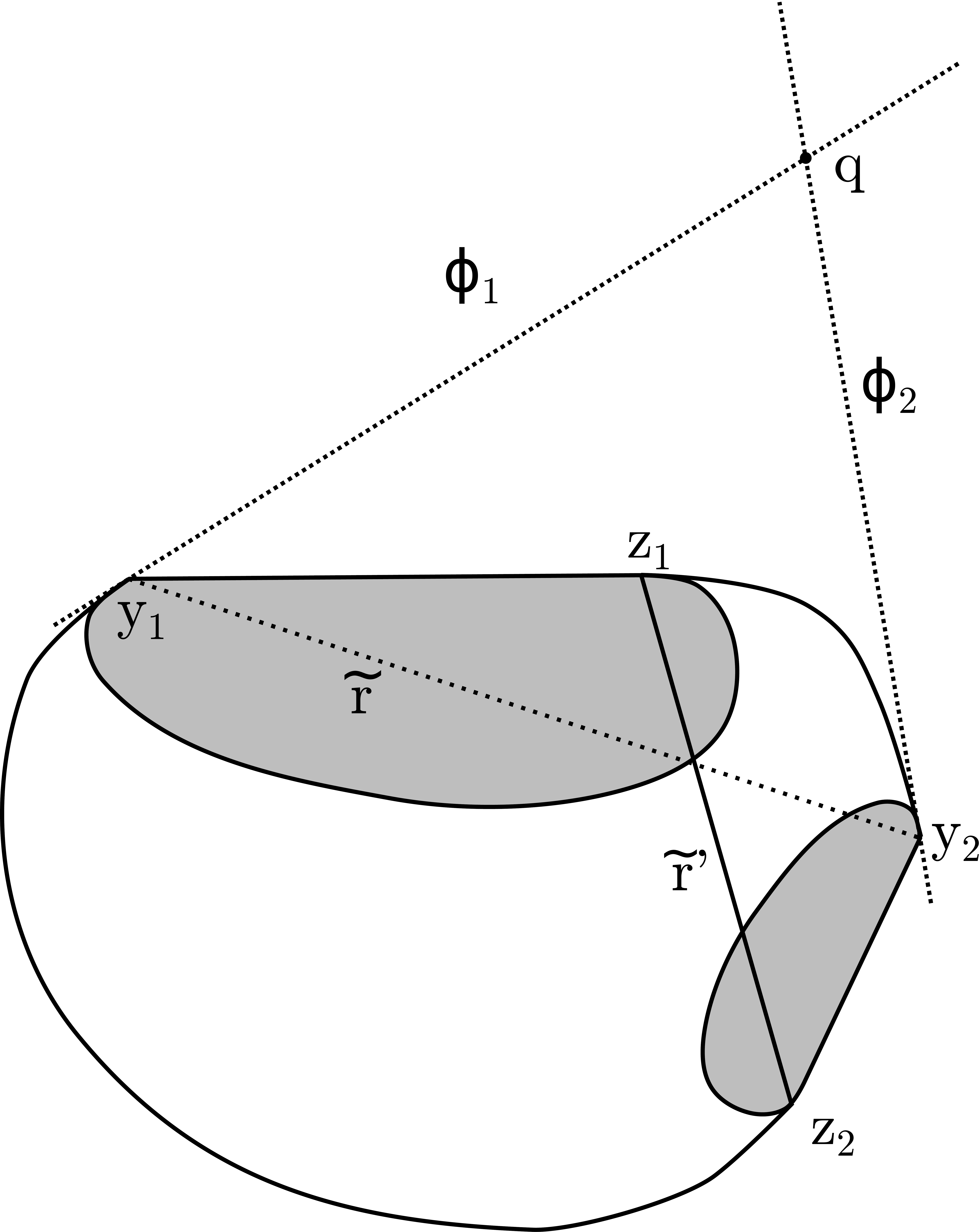}
\caption{A schematic of the domain $\Omega_{t_0}$ and covers of the arcs $r$ and $r'$. Type $1$ horoballs are shaded.}
\label{figure: surface double bending}
\end{figure}

It is not difficult to construct such a deformation for surfaces, because the necessary codimension-one submanifold is an arc, and thus has trivial fundamental group. For example, if we choose a finite-volume hyperbolic surface $\Sigma$ with more than one cusp, and an embedded arc $r$ between two cusps $P_1$ and $P_2$ in $\Sigma$, we may explicitly describe the deformation and its effect on the structures of $P_1$ and $P_2$.

Firstly, we perform the bending described in Lemma \ref{lemma: local bending} along the arc $r$ for some non-zero parameter $t_0$, which changes the structure of $P_1$ and $P_2$ to type $1$. Now, consider one connected cover of the arc $r$ in the domain $\Omega_{t_0}$ (the domain preserved by the bent representation). It meets $\partial \Omega_{t_0}$ in a pair of points, $y_1$ and $y_2$ and each of these points is an endpoint of a segment in $\partial \Omega_{t_0}$, which are the boundaries of the type $r$ model cusp domains covering $P_1$ and $P_2$. 

Call the other two endpoints of these boundary segments $z_1$ and $z_2$, and let $r'$ be the image of the arc between them in the surface. The points $y_1$ and $y_2$ each have a segment of supporting hyperplanes. One of the endpoints of this segment includes $z_i$, the opposite endpoint does not. Call $\phi_i$ the supporting hyperplane through $y_i$ that is an endpoint of the segment of supporting hyperplanes through $y_i$ and does not include $z_i$. Figure \ref{figure: surface double bending} depicts the domain $\Omega_{t_0}$ and the relevant hyperplanes.

The hyperplanes $\phi_1$ and $\phi_2$ meet at some point $q$ outside of $\Omega_{t_0}$. The matrix pointwise fixing the subspace spanned by $\tilde{r}'$ and with $\lambda$-eigenspace at $q$ is an example of the bending matrix in $Z'(H)$ from the above calculation, and so any such surface admits a bending deformation to a pair of type $2$ cusps. That is, after bending along $r$, bend again along $r'$, changing the structures at $P_1$ and $P_2$ to type $2$. 

While this demonstrates a multitude of examples of type $2$ cusps in cusped surfaces, in any higher dimension it is not obvious how to find a suitable manifold, submanifold pair which permits this bending.
Note that the two bending deformations in this construction do not commute. Indeed, bending along either $r$ or $r'$ makes the other curve cease to be totally geodesic. In particular, if $h=dev_0 \circ (dev_{t_0}^{-1})$, then certainly $h(\tilde{r}')$ is not a totally geodesic submanifold in the initial hyperbolic structure on $\Sigma$. This is indicative of the difficulty in finding an embedded totally geodesic submanifold along which to bend in the higher dimensional examples constructed earlier in this paper. That is, it is not clear what submanifold one should look for in the arithmetic hyperbolic manifold which we deform.

\section{Acknowledgements}
The author would like to thank his advisor, Jeff Danciger for encouragement. In addition, he extends his gratitude to Sam Ballas for his helpful conversation and willingness to exchange ideas, and Alan Reid for his advice and guidance. The author acknowledges support from U.S. National Science Foundation grants DMS 1107452, 1107263, 1107367 "RNMS: Geometric Structures and Representation Varieties" (the GEAR Network), as well as support by NSF Research and Training Grant DMS-1148940.

\newpage

\bibliographystyle{amsplain}
\bibliography{iterated_bending}

\nocite{Scott2}

\end{document}